\documentclass[11pt,bezier]{article}
\usepackage[usenames,dvipsnames]{xcolor}
\usepackage{amsmath,amssymb,amsfonts,euscript,graphicx,amsthm,mathrsfs}
\usepackage[colorlinks=true,citecolor=blue,linkcolor=red]{hyperref}
\usepackage[nameinlink]{cleveref}
\usepackage[all]{xy}
\usepackage{multicol}
\usepackage{enumerate}
\usepackage{tikz}

\usepackage[top=3cm, bottom=3cm, left=2.5cm, right=3cm]{geometry}
  \newtheorem{The}{Theorem}[section]
  \newtheorem{Pro}[The]{Proposition}
  \newtheorem{Lem}[The]{Lemma}
  \newtheorem{Cor}[The]{Corollary}
  \newtheorem{Def}[The]{Definition}
  
  \newtheorem{Rem}[The]{Remark}
  
  \newtheorem{Examp}[The]{Example}

\newtheorem{Ques}[The]{Question}

\newcommand{\Z}{\mathcal{Z}}

\let\oldproofname=\proofname
\renewcommand{\proofname}{\textit{\rm\bf\oldproofname}}

\title{\bf\Large Weakly uniserial dimension of modules
\author { {\bf N. Shirali$^{\rm a,}$\thanks{Corresponding author.}}~,   {\bf R. Beyranvand$^{\rm b}$}   {and \bf S. Shirzadi$^{\rm b}$} \\
{\small{ $^{{\rm a}}$ Department of Mathematics, Shahid Chamran university of Ahvaz, Ahvaz, Iran}}\\
{\small{ $^{{\rm b}}$ Department of Mathematics, Faculty of Science, Lorestan University,  Khorramabad, Iran}}\\
   {\small{shirali\_n@scu.ac.ir}}\vspace{-1mm}\\
   {\small{beyranvand.r@lu.ac.ir}}\vspace{-1mm}\\
    {\small{shirzadi.sa@fs.lu.ac.ir}}\vspace{-1mm}\\
   {\small{}}}}
  \date{}

\begin{document}
\maketitle
\begin{abstract}
Recall that a module is called weakly uniserial if its submodules are comparable regarding
embedding. Weakly uniserial modules are a nontrivial generalization of uniserial modules.
In this paper we define and study a new dimension, which  measure how
far a module deviates from being weakly uniserial. We call this dimension, weakly uniserial dimension.
Also, we define and study monoartinian (mononoetherian)  modules.
We say that an $R$-module $M$ is monoartinian (mononoetherian)  if in every descending (ascending)
chain of submodules of $M$, except probably a finite number, each module in chain embedded in 
 the next (previous) one.  We show that a module  has weakly uniserial dimension if and only if it is  monoartonian.
\end{abstract}
 
\section{\hspace{-6mm}. Introduction}
Recall that an $R$-module $M$ is said to be \textit{uniserial} if its submodules are linearly ordered by
inclusion. An $R$-module $M$ is called \textit{serial} if it is a direct sum of uniserial modules. A \textit{left} (resp.,
\textit{right}) \textit{uniserial} ring is a ring which is uniserial as a left (resp., right) module. Also, a ring $R$ is
called uniserial (resp., serial) if it is both a left and a right uniserial (resp., serial) ring. 
Dimensions like Gelfand, Krull, Goldie have an important role in the study
of theory of rings and modules. In \cite{Naz}, Z. Nazemian et al.,  introduced \textit{uniserial dimension} of modules. In order to defined uniserial dimension for modules over a ring $R$, they first define, by transfinite induction, classes $\zeta_{\alpha}$ of $R$-modules for all ordinals $\alpha \geq 1$.
 To start with, let $\zeta_{1}$ be the class of non-zero uniserial modules.
 Next, consider an ordinal $\alpha > 1$; if $\zeta_{\beta}$ has been defined for all ordinals $\beta < \alpha$,
let $\zeta_{\alpha}$ be the class of those $R$-modules $M$ such that, for every submodule $N < M$, where $M/N \ncong M$, we have $M/N \in \bigcup_{\beta <\alpha} \zeta_{\beta}$.
If an $R$-module $M$ belongs to some  $\zeta_{\alpha}$, then the least such $\alpha$ is the uniserial
dimension of $M$, denoted ${\rm u.s.dim}(M)$. For $M = 0$,   ${\rm u.s.dim}(M)=0$ and if  $M \neq 0$ and $M$ does not belong to any $\zeta_{\alpha}$, then  $M$ has no uniserial dimension. 
they have shown that for a ring $R$ and an ordinal number $\alpha$,
there exists an $R$-module of uniserial dimension $\alpha$. It has been shown that  
a commutative ring $R$ is Noetherian (resp. Artinian) if and only
if every finitely generated $R$-module has (resp. finite) uniserial
dimension. Also, they have shown that every right $R$-module has
uniserial dimension if and only if the free right $R$-module  $\oplus_{i=1}^{\infty}R$
has uniserial dimension if and only if $R$ is a semisimple Artinian
ring.

In \cite{Gho}, \textit{Couniserial dimension} of modules, bas been  introduced  by A. Ghorbani et al.
 Couniserial dimension is a measure of how far a module deviates from being uniform. Note that if a module $M$ is isomorphic to all its non-zero submodules, then $M$ must be uniform. 
In order to defined couniserial dimension for modules
over a ring $R$, they first defined, by transfinite induction, classes $\zeta_{\alpha}$ of $R$-modules for all
ordinals $\alpha \geq 1$. To start with, let  $\zeta_{1}$ be the class of all uniform
modules. Next, consider an ordinal $\alpha > 1$; if $\zeta_{\beta}$ has been defined for all ordinals
$\beta < \alpha$, let $\zeta_{\alpha}$ be the class of those $R$-modules $M$ such that for every non-zero submodule $N$ of $M$, where $N \ncong M$, we have $N \in \bigcup_{\beta <\alpha} \zeta_{\beta}$. If an $R$-module $M$ belongs to some $\zeta_{\alpha}$, then the least such $\alpha$ is called the couniserial dimension of $M$, denoted by
${\rm c.u.dim}(M)$. For $M = 0$,  ${\rm c.u.dim}(M)=0$. If a non-zero module $M$ does
not belong to any $\zeta_{\alpha}$, then  $M$ has no couniserial dimension.  They  have   proved  that each module with finite  couniserial dimension  has finite uniform dimension.
Also,  it has been shown that every module of finite length has couniserial dimension and its value lies between the uniform dimension and the length of the module. Modules with countable couniserial dimension are shown to
possess indecomposable decomposition. If the maximal right quotient ring of a semiprime right non-singular ring $R$ has a couniserial dimension as an $R$-module, then $R$ is a semiprime right Goldie ring. They have shown that
 all right $R$-modules have couniserial dimension if and only if $R$ is a semisimple artinian ring.

In \cite{Shi}, we introduced and studied the class of weakly uniserial modules and rings as
a nontrivial generalization of uniserial modules and rings. An $R$-module $M$ is called \textit{weakly uniserial} if  for any two submodules $N$ and $K$ of $M$, $N$ is embedded in $K$ or $K$ is embedded in $N$. Also, a \textit{right} (resp., \textit{left}) \textit{weakly uniserial} ring is a ring which is weakly uniserial as a right (resp., left) $R$-module. 
In \cite{Shi2},
in addition to providing more properties of these modules, we introduce and investigate \textit{weakly serial} modules that are a direct sums of weakly uniserial modules. 

In this paper we define and study \textit{weakly uniserial dimension} of modules. It is an ordinal-valued invariant
that measures how far a module is from being weakly uniserial.
Throughout this paper, all rings have identity elements and all modules are unitary right modules.
For a ring $R$, the Jacobson radical and the right singular  ideal of $R$ are denoted by
${\rm J}(R)$ and $ \Z(R_{R})$, respectively. For a module $M$, the socle, the
injective hull and the singular submodule of $M$ are denoted by ${\rm Soc}(M)$, ${\rm E}(M)$ and $\Z(M)$, respectively. Also for a subset $X$ of $R$, the right (resp., left) annihilator
of $X$ in $R$ is denoted by ${\rm r.Ann}_{R}(X)$ (resp., ${\rm l.Ann}_{R}(X)$). By $K \subseteq M$ we usually mean that $K$ is a submodule of $M$ and the notation $N \subseteq_{e} M $ means that $N$ is an essential submodule of $M$. For  any two $R$-modules $M, N$ if there exists an $R$-monomorphism from $M$ to $N$, we write $ M\rightarrowtail N$  otherwise we write $ M\not\rightarrowtail N$.

In order to define weakly uniserial dimension for modules over a ring $R$, we first define, by transfinite induction, classes $\xi_{\alpha}$ of $R$-modules for all ordinals   $\alpha \geq 1$. To start with, let $\xi_{1}$  be the class of non-zero weakly uniserial modules. Next, consider an ordinal $\alpha > 1$; if  $\xi_{\beta}$  has been defined for all ordinals $\beta < \alpha$, let  $\xi_{\alpha}$   be the class of those $R$-modules $M$ such that, for every submodule $N \subsetneq M$, where $M \not\rightarrowtail N$, we have  $N \in \bigcup_{\beta < \alpha} \xi_{\beta}$. If an $R$-module $M$ be longs to some  $\xi_{\alpha}$, then the least such $\alpha$ is the weakly uniserial dimension of $M$, denoted ${\rm w.dim}(M)$. For $M=0$, we define ${\rm w.dim}(M)=0$. If $M$ is non-zero and $M$ does not belong to any  $\xi_{\alpha}$, then we say that ''${\rm w.dim}(M)$ is not defined'', or that ''$M$ has no weakly uniserial dimension''.

The paper is organized as follows. In Section $2$, we give some basic properties of  modules with weakly
uniserial dimension. It is shown  that a module  has weakly uniserial dimension if and only if it is monoartinian (see \Cref{1}).
We showed that a semisimple module has  weakly uniserial dimension if and only if it is a finite direct sum of homogeneous semisimple modules (see \Cref{semisimple}). In \Cref{injective z-module}, we prove that every injective $\mathbb{Z}$-module with weakly uniserial dimension is isomorphic to a finite direct sum of $\mathbb{Z}_{p^{\infty}}$, where $p$ belongs to the set of prime numbers.
In Section $3$, we study  notions of  monoartinian and mononoetherian modules and rings. It is shown that being monoartinian (mononoetherian) is a Morita invariant  property and if  $R$ is a right  monoartinian  (mononoetherian) ring, then $R[x]$ is right monoartinian  (mononoetherian) ring  (see \Cref{morita} and \Cref{poly} ).  Also in \Cref{Goldie},  we showed that every commutative semiprime monoartinian ring is a Goldie ring. Over a commutative ring  $R$  every mononoetherian  $R$-module has finite uniform dimension (see \Cref{u.dim}). Also,
among the other  properties of monoartinina (mononoetherian) modules and rings, we anwer the following question:
\begin{Ques}
 Which rings R have the property that every  right R-module is monoartinian?
\end{Ques}

\section{\hspace{-6mm}. Weakly uniserial dimension}
In this section, we give some properties of modules and rings with weakly uniserial dimension. 
Recall that  a right $R$-module $M$ is said to be {\it weakly uniserial} if for any two submodules $N$ and $K$ of $M$, $ N\rightarrowtail K$ or $ K\rightarrowtail N$.  A {\it left} (resp., {\it right}) {\it weakly uniserial ring} is a ring which is weakly uniserial as a left (resp., right) module. As usual, a ring $R$ is called {\it weakly uniserial} if it is both a left and a right weakly uniserial ring (see \cite{Shi}).

\begin{Def}
{\rm
In order to define \textit{weakly uniserial dimension} for modules over a ring $R$, we first define, by transfinite induction, classes $\xi_{\alpha}$ of $R$-modules for all ordinals   $\alpha \geq 1$. To start with, let $\xi_{1}$  be the class of non-zero weakly uniserial modules. Next, consider an ordinal $\alpha > 1$; if  $\xi_{\beta}$  has been defined for all ordinals $\beta < \alpha$, let  $\xi_{\alpha}$   be the class of those $R$-modules $M$ such that, for every submodule $N \subsetneq M$, where $M \not\rightarrowtail N$, we have  $N \in \bigcup_{\beta < \alpha} \xi_{\beta}$. If an $R$-module $M$ be longs to some  $\xi_{\alpha}$, then the least such $\alpha$ is the weakly uniserial dimension of $M$, denoted ${\rm w.dim}(M)$. For $M=0$, we define ${\rm w.dim}(M)=0$. If $M$ is non-zero and $M$ does not belong to any  $\xi_{\alpha}$, then we say that ''${\rm w.dim}(M)$ is not defined'', or that ''$M$ has no weakly uniserial dimension''.
}
\end{Def}

\begin{Rem}\label{rem1}
{\rm
We make the convention that a statement ''${\rm w.dim}(M) = \alpha$”  will mean
that the weakly uniserial dimension of $M$ exists and equals $\alpha$. By the definition of weakly uniserial
dimension, if $M$ has weakly uniserial dimension and $N$ is a submodule of $M$, then $N$
has weakly uniserial dimension and ${\rm w.dim}(N) \leq {\rm w.dim}(M)$. Moreover, if $M$ is not
weakly uniserial and ${\rm w.dim}(M) = {\rm w.dim}(N)$, where $N$ is a submodule of $M$, then 
$M \rightarrowtail N$.
 On the other hand, since every set of ordinal numbers has supremum, it follows
immediately from the definition that $M$ has weakly uniserial dimension if and only if for
all submodules $N$ of $M$ with $M \not\rightarrowtail N$, ${\rm w.dim}(N)$ is defined. In the latter case, if 
$\alpha= {\rm sup}\lbrace {\rm w.dim}(N) ~| ~N \subsetneq M,~ M \not\rightarrowtail N \rbrace$, then
$ {\rm w.dim}(M)  \leq \alpha + 1$.
}
\end{Rem}

\begin{Lem}\label{وسطی}
If $M$ is an $R$-module and ${\rm w.dim}(M) = \alpha$, then for any $0 \leq \beta \leq \alpha$,
there exists a submodule $N$ of $M$ such that ${\rm w.dim}(N)=\beta $.
\end{Lem}

\begin{proof}
The proof is by transfinite induction on ${\rm w.dim}(M) = \alpha$. The case $\alpha =1$ is
clear. Let $\alpha > 1$ and $0 \leq \beta <\alpha$. Then, by \Cref{rem1}, there exists a submodule
$K$ of $M$ such that  $M \not\rightarrowtail K$ and $\beta \leq {\rm w.dim}(K)$. Then $\beta \leq {\rm w.dim}(K) < \alpha$ and so by induction hypothesis, there exists a submodule $N$ of $K$ such that ${\rm w.dim}(N)=\beta$.	
\end{proof}

\begin{Def}
{\rm
Let $R$ be a ring and $M$ be a right $R$-module. We say that $M$ is \textit{monoartinian}
if, for every descending chain $M \supseteq M_{1} \supseteq M_{2} \supseteq \cdots $ of submodules of $M$, there exists an index $n \geq 1$ such that $M_{n}$ is embedded in $M_{i}$ for every $ i \geq n$.
}
\end{Def}

The following theorem provides a working criterion for a module to have  weakly uniserial dimension.

\begin{The}\label{1}
A module  has weakly uniserial dimension if and only if it is monoartinian.
\end{The}

\begin{proof}
$(\Rightarrow)$ Let $M$ be a module with weakly uniserial dimension and
 $M_{1} \supseteq M_{2} \supseteq \cdots$ be a descending chain of submodules of it. Put
$\gamma = {\rm inf} \lbrace{\rm w.dim}(M_{i})~ | ~i \geq 1\rbrace$. So $\gamma = {\rm w.dim}(M_{n})$ for some $n \geq 1$. Then $M_{n} \rightarrowtail M_{k}$, for any $k \geq n$ or there exists $l \geq n$ such that $M_{n} \not\rightarrowtail M_{l}$. If  $M_{n} \not\rightarrowtail M_{l}$, then by definition, ${\rm w.dim}(M_{l}) \lneq {\rm w.dim}(M_{n})=\gamma$, a contradiction. Therefore,  $M_{n} \rightarrowtail M_{k}$,  for any $k \geq n$ and so $M$ is monoartinian.

$(\Leftarrow)$ Assume that $M$ does not have weakly uniserial dimension. Then $M$ is not weakly uniserial and so there
exists a submodule $M_{1}$ of $M$ such that $M\not\rightarrowtail M_{1}$ and $M_{1}$ does not have weakly uniserial
dimension, by \Cref{rem1}. So there exists a submodule $M_{2}$ of $M_{1}$ such that
 $M_{1}\not\rightarrowtail M_{2}$ and $M_{2}$ does not have weakly uniserial
dimension. Continuing in this manner, we obtain a descending chain of submodules 
$M_{1} \supsetneq M_{2} \supsetneq \cdots$, 
such that for every $i \geq 1$, $M_{i}\not\rightarrowtail M_{i+1}$, a contradiction. This
completes the proof.
\end{proof}

A right $R$-module $M$ is called \textit{duo} provided every submodule of $M$ is fully invariant \cite{Ozc}.

\begin{Lem}\label{injective duo}
{\rm (\cite[page 185]{Bum})} Let $B$ be a fully invariant submodule of an injective $R$-module $E$
and $A \subseteq E$. If there is an $R$-monomorphism from $A$ to $B$, then $A \subseteq B$.
\end{Lem}

\begin{Cor}
Let  $M$  be an injective duo right $R$-module. Then the following statements hold:\\
\indent{{\rm{(a)}}} $M$ has weakly uniserial dimension if and only if  it is artinian.\\
\indent{{\rm{(b)}}} If  $M$  has weakly uniserial dimension, then it  has a uniserial submodule. Moreover, $M$
 is \indent\indent a finite direct sum of  indecomposable  modules.
\end{Cor}

\begin{proof}
(a).  Assume that $M$  is an injective duo  right  $R$-module with weakly uniserial dimension. Then for every descending chain 
$ M_{1} \supseteq M_{2} \supseteq \cdots, $
of submodules of  $M$, there exists $n \geq 1$ such that $M_{n} \rightarrowtail  M_{i}$ for any  $i\geq n$, by \Cref{1}. Now,  \Cref{injective duo} implies that  $M_{n} \subseteq  M_{i}$, and so $M$  is artinian. 

(b). Since $M$ has weakly uniserial dimension, \Cref{وسطی} implies that $M$ has a weakly uniserial submodule, say $N$. Then for any two submodules $N_{1}$ and $N_{2}$ of  $N$ we have  $N_{1} \rightarrowtail N_{2}$  or     $N_{2} \rightarrowtail N_{1}$. Now, \Cref{injective duo} complete the proof. Moreover part, follows from (a) and \cite[Proposition 19.20]{Lam1}.
\end{proof}

Recall that two modules $M$ and $N$  are orthogonal if they have no non-zero isomorphic submodules; whereas $M$ is parallel to $N$  if no non-zero submodule of $M$ is orthogonal  to $N$  and symmetrically no non-zero submodule of $N$  is orthogonal  to  $M$. A module $M$  is called atomic if all its non-zero submodules are parallel to each other. It is easy to see that, weakly uniserial modules are atomic.
Also an $R$-module  $M$  is said to be  $p$-Artinian if for every descending chain 
$M_{1} \supseteq M_{2} \supseteq \cdots$
of submodules of  $M$, there exists $n \geq 1$ such that $M_{n}$  is parallel to $M_{i}$ for any   $i\geq n$.

\begin{Cor}\label{p-art}
\indent{{\rm{(a)}}} Every module with weakly uniserial dimension  is  $p$-Artinian.\\
\indent{{\rm{(b)}}} Every Artinian module has weakly uniserial dimension.
\end{Cor}

\begin{proof}
(a). Assume that $M$ is a right $R$-module with weakly uniserial dimension and 
$M_{1} \supseteq M_{2} \supseteq \cdots$
is a descending chain of submodules of $M$.  By \Cref{1}, there exists $n \geq 1$ such that  $M_{n} \rightarrowtail M_{k}$, for all $k \geq n$.  Since $M_{k} \subseteq M_{n}$,  for any $ 0 \neq M'_{k} \subseteq M_{k}$, $M_{n}$ is not orthogonal to $M'_{k}$. Also since for every $ 0 \neq M'_{n} \subseteq M_{n}$, $M'_{n} \rightarrowtail M_{k}$ we conclude that 
$M'_{n}$ is  not orthogonal to $M_{k}$. Therefore,   $M_{n}$ is parallel to   $M_{k}$ for any  $k \geq n$ and so $M$ is a  $p$-Artinian module.

(b). The proof is clear by \Cref{1}.
\end{proof}

\begin{Examp}
{\rm
By \cite[Example 5.5]{Shirali}, the $\mathbb{Z}$-module  $\oplus_{i > 0} \mathbb{Z}_{p^{i}}$ is not p-artinian. Then  \Cref{p-art} implies that it has not weakly uniserial dimension.
}
\end{Examp}

\begin{Examp}\label{Q}
{\rm
$\mathbb{Q} $ has not weakly uniserial dimension as a $\mathbb{Z}$-module. To see this,
let $\lbrace p_{1}, p_{2}, \ldots \rbrace$ be the set of prime numbers in $\mathbb{Z}$. For each $i$, set 
$A_{i}=\lbrace  r/s ~|~ r, s \in \mathbb{Z} ~and ~p_{j} \nmid s ~for~ all ~1 \leq j \leq i \rbrace$.  It is clear that 
$A_{1} \supseteq A_{2} \supseteq \cdots$, is a descending chain
of $\mathbb{Z}$-submodules of $\mathbb{Q}$. Moreover for any $n \geq 1$, $A_{n}$ cannot be embedded in $A_{n+1}$, because one is
$p_{n}$-divisible and the other is not.
}
\end{Examp}

\begin{Pro}\label{injective z-module}
Every injective $\mathbb{Z}$-module with weakly uniserial dimension is isomorphic to a finite direct sum of $\mathbb{Z}_{p^{\infty}}$, where $p$ belongs to the set of prime numbers. 
\end{Pro}

\begin{proof}
Assume that $M$ is an  injective $\mathbb{Z}$-module with weakly uniserial dimension.  Since  $\mathbb{Z}$ is a Noetherian ring, $M$ is a direct sum of injective indecomposable  $\mathbb{Z}$-modules. But an injective
indecomposable  $\mathbb{Z}$-module is isomorphic to either ${\rm E}(\mathbb{Z}) \cong \mathbb{Q}$ or ${\rm E}(\mathbb{Z}/p\mathbb{Z}) \cong \mathbb{Z}_{p^{\infty}}$, where
$p$ is a prime number. Thus
$M \cong \mathbb{Q}^{(\lambda)} \oplus (\oplus_{p \in P} \mathbb{Z}_{p^{\infty}})$,
 where $P$ is the set of prime numbers. Now, by \Cref{Q}, \Cref{rem1} and \Cref{1} we conclude that $M$ is isomorphic to a finite direct sum of $\mathbb{Z}_{p^{\infty}}$, where $p \in P$.
\end{proof}

\begin{Def}
{\rm
We say that a module is anti weakly uniserial if its incomparable submodules are not comparable regarding embedding.
}
\end{Def}

\begin{Pro}
Let $M = M_{1} \oplus M_{2}$ be an anti weakly uniserial right $R$-module with weakly uniserial dimension. Then 
${\rm w.dim}(M) \geq {\rm w.dim}(M_{1}) +{\rm w.dim}(M_{2})$.
\end{Pro}

\begin{proof}
Assume that ${\rm w.dim}(M_{2})=\alpha$. We use transfinite induction on $\alpha$. Not that since $M \not\rightarrowtail M_{1}$, we have ${\rm w.dim}(M_{1})< {\rm w.dim}(M)$ and so  ${\rm w.dim}(M_{1}) +1 \leq {\rm w.dim}(M)$. So the case $\alpha =1$ is clear.
 Thus, suppose $\alpha > 1$ and for every anti weakly uniserial  right $R$-module $L$ of  weakly uniserial dimension less than $\alpha$, 
 ${\rm w.dim}(M_{1}) + {\rm w.dim}(L) \leq {\rm w.dim}(M_{1} \oplus L)$. 
 There are two possibilities: $\alpha$ is a successor ordinal, or a limit ordinal. 
 
 Suppose first that $\alpha$ is a successor ordinal. Then there exists an ordinal number $\gamma$ such that  $\alpha = \gamma +1$ and so $0 < \gamma <\alpha$. Hence by \Cref{وسطی}, there exists a non-zero proper submodule $K$ of $M_{2}$ such that 
 ${\rm w.dim}(K)= \gamma<\alpha$. Since $M$  is anti weakly uniserial,  $M_{1}\oplus K$ is also anti weakly uniserial and so by the induction hypothesis we have
$${\rm w.dim}(M_{1})  + \gamma = {\rm w.dim}(M_{1}) + {\rm w.dim}(K)  \leq {\rm w.dim}(M_{1} \oplus K). $$
Since $M \not\rightarrowtail M_{1} \oplus K$, ${\rm w.dim}(M_{1} \oplus K) < {\rm w.dim}(M)$  and hence
$${\rm w.dim}(M_{1}) +\gamma +1={\rm w.dim}(M_{1})+{\rm w.dim}(K)+1 \leq {\rm w.dim}(M).$$
Therefore, ${\rm w.dim}(M_{1}) + {\rm w.dim}(M_{2}) \leq {\rm w.dim}(M)$. 

If $\alpha$ is a limit ordinal and $1 \leq \beta < \alpha$ then, by \Cref{وسطی}, there exists a non-zero proper
submodule $K$ of $M_{2}$ such that $\beta \leq {\rm w.dim}(K)$. Then by induction hypothesis
$${\rm w.dim}(M_{1}) + \beta \leq {\rm w.dim}(M_{1}) +{\rm w.dim}(K) \leq {\rm w.dim}(M_{1} \oplus K) <{\rm w.dim}(M)$$
Therefore, ${\rm w.dim}(M_{1}) + \alpha = {\rm sup} \lbrace {\rm w.dim}(M_{1}) + \beta ~|~ \beta  < \alpha \rbrace \leq {\rm w.dim}(M)$.
\end{proof}

\begin{Rem}
{\rm
The condition anti weakly uniserial  in above proposition is necessary. For example consider two $\mathbb{Z}$-modules
$M=\oplus_{1}^{\infty}\mathbb{Z}_{p}$ and $N=\mathbb{Z}_{p}$. Then $M \cong M \oplus N$, but ${\rm w.dim}(M)=1$ and so ${\rm w.dim}(M) \ngeq {\rm w.dim}(M)+{\rm w.dim}(N)$.
}
\end{Rem}

\begin{Cor}
Let $M$  be an anti weakly uniserial module  and $N$ be a submodule of it.  If  ${\rm w.dim}(N)={\rm w.dim}(M)$,  then  $N \subseteq_{e} M$.
\end{Cor}
\begin{proof}
Assume, to the contrary, that $N$ is not an essential  submodule of $M$. Then there exists a non-zero submodule $K$  of $M$ such that $N \cap K=0$  and so $N \oplus K \subseteq M$.  Hence by above proposition,
${\rm w.dim}(N) < {\rm w.dim}(N)+{\rm w.dim}(k) \leq {\rm w.dim}(N \oplus k) \leq M$,
a contradiction.
\end{proof}

\begin{Rem}
{\rm
The converse of above corollary  is not true in general. For example....
}
\end{Rem}

\begin{Pro}
Let $M$ be a right $R$-module of finite length. Then the following statements hold:\\
\indent{{\rm(a)}} If $N$ is a submodule of $M$, then ${\rm w.dim}(M/N) \leq {\rm w.dim}(M) $.\\
\indent{{\rm(b)}} ${\rm w.dim}(M) \leq {\rm length}(M)$.
\end{Pro}

\begin{proof}
(a). The proof is by induction on $n$, where ${\rm length}(M) = n$. The case $n = 1$ is clear. Now, let $n > 1$ and assume that the assertion is true for all modules with length less than $n$. If $N$ is a non-zero submodule of $M$, then the
${\rm length}(M/N) < n$. Thus for every proper submodule $K/N$ of $M/N$, by induction,
${\rm w.dim}(K/N) \leq {\rm w.dim}(K) < {\rm w.dim}(M)$. Now, \Cref{rem1} implies that
${\rm w.dim}(M/N) \leq {\rm w.dim}(M)$.

(b). The proof is by induction on ${\rm length}(M) = n$. The case $n = 1$ is clear. Now if
$n > 1$ and $K$ is a proper submodule of $M$, then by assumption,
$ {\rm w.dim}(K) \leq {\rm length}(K) < {\rm length}(M)$. Thus by \Cref{rem1}, ${\rm w.dim}(M) \leq {\rm length}(M)$.
\end{proof}
\begin{Rem}
{\rm
The condition finite length  in above proposition is necessary. for instance, consider   $\mathbb{Z}$-modules  $M=\mathbb{Z}$ and $N=6\mathbb{Z}$. Then ${\rm w.dim}(M/N)=2$ and ${\rm w.dim}(M)=1$ and so ${\rm w.dim}(M/N) \nleq {\rm w.dim}(M)$.
}
\end{Rem}

\begin{Pro}\label{E(s)}
Let  $S$ be a  simple right  $R$-module. If  $\oplus_{1}^{\infty} {\rm E}(S)$ has weakly uniserial dimension, then $S$ is   injective or  singular.
\end{Pro}

\begin{proof}
Consider the following descending chain  of submodules of $\oplus_{1}^{\infty} {\rm E}(S)$:
$$
S\oplus(\oplus_{2}^{\infty} {\rm E}(S)) \supseteq S \oplus S \oplus (\oplus_{3}^{\infty} {\rm E}(S)) \supseteq \cdots. 
$$
Then, by \Cref{1},  there exists $n \geq 1$ such that  
$S^{(n)}\oplus(\oplus_{n+1}^{\infty} {\rm E}(S)) \overset{f}\rightarrowtail S^{(k)}\oplus(\oplus_{k+1}^{\infty} {\rm E}(S)) $, for any $k\geq n$. Thus we have the following nonzero  $R$-homomorphism:
$${\rm E}(S) \overset{i}\hookrightarrow S^{(n)}\oplus(\oplus_{n+1}^{\infty} {\rm E}(S))  \overset{f}\rightarrowtail S^{(k)}\oplus(\oplus_{k+1}^{\infty} {\rm E}(S)) \overset{\pi}\rightarrow S,$$ 
where $\pi$ is a suitable canonical projection. By set $g=\pi f i$ we have $\frac{{\rm E}(S)}{{\rm ker}(g)} \cong S$. If ${\rm ker}(g)=0$, then ${\rm E}(S) \cong S$  and if, ${\rm ker}(g) \neq 0$ we conclude that $S$  is a singular module.
\end{proof}

\begin{Pro}\label{semisimple}
Let $M$ be a semisimple right $R$-module. Then $M$ has weakly uniserial dimension, if and only if $M$ is a finite direct sum of homogeneous semisimple $R$-modules.
\end{Pro}
\begin{proof}
($\Rightarrow$)
If $M$ is not a finite direct sum of homogeneous semisimple $R$-modules, then there exists a summand $N$ of $M$ such that $N=\oplus_{i=1}^{\infty}M_{i}$, 
where the $M_{i}$'s are pairwise non-isomorphic simple right R-modules. Now consider thea scending chain
$\oplus_{i=1}^{\infty}M_{i} \supseteq \oplus_{i=2}^{\infty}M_{i} \supseteq \cdots,$
of submodules of $N$. Using \Cref{1}, we  see that $N$ doesn’t have weakly uniserial dimension. Therefore, $M$ doesn’t have weakly uniserial dimension, a contradiction.

($\Leftarrow$)
Let $M$ be a finite direct sum of homogeneous semisimple $R$-modules. Then there exist pairwise non-isomorphic simple $R$-modules $S_{1}, \ldots, S_{m}$ and index sets $I_{1}, \ldots, I_{m}$ such that 
$M=(\oplus_{I_{1}}S_{1}) \oplus \cdots \oplus (\oplus_{I_{m}}S_{m})$.
  Assume to the contrary that $M$ doesn’t have weakly uniserial dimension. Let  $A_{1}$ be the set of cardinal numbers $c$ such that  there exist index sets $J_{1}, \ldots, J_{m}$ such that 
  $|J_{1}|=c$ and $N=(\oplus_{J_{1}}S_{1}) \oplus \cdots \oplus (\oplus_{J_{m}}S_{m})$
doesn’t have weakly uniserial dimension. Then $A_{1}$ is non-empty and so has the minimum element $c_{1}$. We can inductively construct a sequence of cardinal numbers $c_{1}, \ldots, c_{m}$ in terms of each $c_{i}$ is minimum element in the set $A_{i}$ of cardinal numbers $c$ such that there exist index sets $J_{1}, \ldots, J_{m}$ with the property that 
$|J_{k}|= c_{k}$ for each $1 \leq k \leq i-1$, $|J_{i}|=c$ and
$N=(\oplus_{J_{1}}S_{1}) \oplus \cdots \oplus (\oplus_{J_{m}}S_{m})$
doesn’t have weakly uniserial dimension. Set
$K=(\oplus_{J_{1}}S_{1}) \oplus \cdots \oplus (\oplus_{J_{m}}S_{m})$
where for each $i$, $|J_{i}|=c_{i}$. Since $K$ doesn’t have weakly uniserial dimension, there exists a submodule $L$ of $K$ such that  $K \not\rightarrowtail L$
and $L$ doesn’t have weakly uniserial dimension. Because $K$ is semisimple, there exist index sets $K_{1}, \ldots, K_{m}$, where $|K_{i}| \leq c_{i}$ and 
$L\cong(\oplus_{K_{1}}S_{1}) \oplus \cdots \oplus (\oplus_{K_{m}}S_{m})$.
 Since  $K \not\rightarrowtail L$,  there exist $i$ such that $|K_{i}| < c_{i}$, a contradiction to the choice of $c_{i}$. 
\end{proof}

\begin{Pro}
Let $M$  be a right  $R$-module with weakly uniserial dimension such that it is not embedded in any of its submodules. Then $M$ is a direct sum of weakly uniserial and indecomposable modules.
\end{Pro}

\begin{proof}
The proof is by transfinite induction on  ${\rm w.dim}(M)=\alpha$.  The case $\alpha = 1$ is clear. Now, let $\alpha > 1$ and assume that the assertion is true for every right  $R$-module $N$ with weakly uniserial dimension less than $\alpha$ that it is not  embedded in any of its submodules. If  $M$ is indecomposable, then we are done. So assume that there exist two non-zero proper submodules $M_{1}$ and $M_{2}$  of $M$  such that  $M=M_{1} \oplus M_{2}$. Since 
$M \not\rightarrowtail M_{i}$,  we have  ${\rm w.dim}(M_{i})<\alpha$ and it is easy to see that $M_{i}$ is not embedded in any of its submodules, for  $i \in \lbrace 1, 2\rbrace$. Therefore, by induction hypothesis, $M$ is a direct sum of weakly uniserial and indecomposable modules.
\end{proof}

\begin{Lem}\label{first every}
Let $R$ be a ring such that every right $R$-module has weakly uniserial dimension. Then the following statements hold:\\
\indent{{\rm(a)}} $R$ has finitely many prime ideal.\\
\indent{{\rm(b)}}  $R$ has finitely many maximal right  ideal.\\
\indent{{\rm(c)}} If $R$ is right noetherian  with nil jacobson radical, then $R$ is a right artinian ring.\\
\indent{{\rm(d)}} If $R$ has only one simple right $R$-module {\rm (}up to isomorphism{\rm )}, then $R$ is a right $V$-ring \indent\indent  or  ${\rm soc}^{2}(R_{R})=0$.
\end{Lem}

\begin{proof}
(a). Assume to the contrary that $R$ has infinitely many prime ideal. Then, by hypothesis, the right $R$-module $M=\oplus_{P} R/p_{i}$ has weakly uniserial dimension, where $P={\rm spec}(R)$. But 
$$\oplus_{P} R/p_{i} \supseteq \oplus_{P \setminus \lbrace p_{1} \rbrace} R/p_{i} \supseteq  \oplus_{P \setminus \lbrace p_{1}, p_{2} \rbrace} R/p_{i} \supseteq  \cdots,$$ is a descending chain of submodules of $M$ that 
$ \oplus_{P \setminus \lbrace p_{1}, \ldots, p_{n} \rbrace} R/p_{i} \not\rightarrowtail  \oplus_{P \setminus \lbrace p_{1}, \ldots, p_{k} \rbrace} R/p_{i}$  for any $k \geq n$, a contradiction.

(b). If $R$ has infinitely many maximal right ideal,  then by hypothesis the right $R$-module $M=\oplus_{i=1}^{\infty} R/m_{i}$ has weakly uniserial dimension, where $m_{i}$'s are distinct maximal right ideals of $R$. But
by \Cref{semisimple}, $M$ does not have weakly uniserial dimension, a contradiction. Therefore,  $R$ has finitely many maximal right  ideal.

(c). By part (b), $R$ is a semilocal ring. Then \cite[Exercise 20.6]{Lam1} completes the proof.

(d). Assume that $S$ is a simple right  $R$-module. Then, by hypothesis, the right $R$-module $\oplus_{1}^{\infty} {\rm E}(S)$  has weakly uniserial dimension.  Now, by \Cref{E(s)} and  hypothesis we conclude that  all simple right $R$-modules are  injective or singular. In first case $R$ is a $V$-ring and in second case  \cite[Lemma 7.2]{Lam} implies that ${\rm soc}^{2}(R_{R})=0$.
\end{proof}

 A \textit{height sequence} $(\alpha_{p})_{p \in \pi}$   is a sequence of non-negative  integers  together with $\infty$, indexed by the elements of $\pi$ (the set of primes of  $\mathbb{Z}$). Let $A$ be a torsion-free  abelian group, $a \in A$  and $p$ is a prime number. Recall that  \textit{$p$-height}  $a$ in $A$, denoted by $h_{p}(a)$, is a non-negative integer $n$ with $a \in p^{n}A \backslash p^{n+1}A$ and $\infty$ if no such $n$ exists. Also   $h_{A}(a) $ (the height sequence  $a$ in $A$)  is the height sequence $(h_{p}(a))_{p \in \pi}$.
For simplicity, we write  $(\alpha_{p})$  instead of  the  height sequence  $(\alpha_{p})_{p \in \pi}$.  
Two height sequences $(\alpha_{p})$ and $(\beta_{p})$ are \textit{equivalent} if $\alpha_{p} = \beta_{p}$ for all but a finite number $p$ and $\alpha_{p} = \beta_{p}$ if either $\alpha_{p}= \infty$ or $\beta_{p}= \infty$. It is easy to see that  this relation is an equivalence relation. An equivalence class $\tau$ of height sequences is called a type, written $\tau=[ \alpha ] $ for some height sequence $\alpha$. Now if $A$ is a torsion-free  abelian group and $a \in A$, then we define the type  $a$ in $A$, to be ${\rm type}_{A}(a)=[h_{A}(a) ] $. If any two nonzero elements of  $A$ have the same type, then the common value being denoted by ${\rm type}(A)$. In this case $A$ is called \textit{homogeneous}, (for more details see \cite{Ar}).

The set of height sequences has a partial ordering given by $(\alpha_{p}) \leq (\beta_{p}) $ if $\alpha_{p} \leq \beta_{p}$ for each $p \in \pi$. Now let ${\rm type}(A)= [(\alpha_{p})]$ and ${\rm type}(B)= [(\beta_{p})]$, where $A$ and $B$ are homogeneous torsion-free abelian groups. Then we define ${\rm type}(A) \leq {\rm type}(B)$ if there exist $(\alpha'_{p}) \in [(\alpha_{p})]$ and $(\beta'_{p}) \in [(\beta_{p})]$ such that $(\alpha'_{p}) \leq (\beta'_{p})$. It is easy to see that if ${\rm type}(A) \leq {\rm type}(B)$ and ${\rm type}(B) \leq {\rm type}(A)$, then ${\rm type}(A) = {\rm type}(B)$. Let $A$ be a torsion-free abelian group. Recall that the \textit{rank} of $A$, denoted by  ${\rm rank}(A)$, is defined  ${\rm dim}_{\mathbb{Q}}(A \otimes_{\mathbb{Z}} \mathbb{Q})$.

\begin{Rem}\label{rank subgroup}
\rm{ A torsion-free abelian group $A$ has rank $1$ if and only if $A$ is isomorphic to a subgroup of $\mathbb{Q}$.
For if ${\rm rank}(A)=1$, then $A \otimes_{\mathbb{Z}} \mathbb{Q} \cong_{\mathbb{Q}} \mathbb{Q}$ and since $A$ is flat, $A$ is isomorphic to a subgroup of $\mathbb{Q}$. On the other hand, if $A$ is a subgroup of $\mathbb{Q}$, then since $\mathbb{Q}_{\mathbb{Z}}$ is flat, $A \otimes_{\mathbb{Z}} \mathbb{Q}$ is embedded  in $\mathbb{Q} \otimes_{\mathbb{Z}} \mathbb{Q} \cong_{\mathbb{Q}} \mathbb{Q}$. Thus ${\rm dim}_{\mathbb{Q}}(A \otimes_{\mathbb{Z}} \mathbb{Q}) =1$ and so ${\rm rank}(A)=1$.} 
 \end{Rem}

\begin{Lem}\label{etype}
{\rm{ (see \cite[Theorem 1.1]{Ar})}} Let $A$ and $B$ be torsion-free abelian groups of ${\rm rank}$ $1$. Then $A$  and $B$ are isomorphic,  if and only if  ${\rm type}(A) = {\rm type}(B)$.
\end{Lem}

\begin{The}
Let $A$ be a torsion-free abelian group of ${\rm rank}$ $1$. Then\\
\indent {\rm (a)} $\rm{W.dim}(A)=1$ if and only if $A\cong \mathbb{Z}$ or  $A\cong G_{p}= \langle \frac{1}{p^{n}} ~ | ~ n \in \mathbb{N}, ~p~ is ~a ~prime~ number \rangle$,  \indent\indent the   subgroup of $\mathbb{Q}$ generated by $\frac{1}{p^{n}}$   where $n \in \mathbb{N}$.\\
\indent {\rm (b)} $\rm{W.dim}(A)=n \geq 2$  if and only if ${\rm type}(A)=[(\alpha_{p})]$, where $\alpha_{p}=0$ for all but a finite  \indent\indent  number $p$ and for $n$ distinct prime number $p$,  $\alpha_{p}=\infty$.
\end{The}

\begin{proof}
(a).($\Rightarrow$). Assume that $\rm{W.dim}(A)=1$. Then $A$ is weakly uniserial and  ${\rm type}(A)=[(\alpha_{p})]$, where $\alpha_{p}=0$ for all but a finite number $p$ and there is at most one $p$ such that $\alpha_{p}=\infty$, by \cite[Theorem 4.6]{Shi2}. Therefore, $A\cong \mathbb{Z}$ or  $A\cong G_{p}$.

($\Leftarrow$). If $A\cong \mathbb{Z}$ or $A\cong G_{p} $, then  \cite[Theorem 4.6]{Shi2} implies that  $A$ is weakly uniserial and so $\rm{W.dim}(A)=1$.

(b).($\Leftarrow$). We prove  the assertion  by induction on $n$. Assume that ${\rm type}(A)=[(\alpha_{p})]$, where $\alpha_{p}=0$ for all but a finite   number $p$ and for two distinct prime number $p$,  $\alpha_{p}=\infty$. If  $G$ is a proper subgroupe of $A$ such that  $A \not\rightarrowtail G$ , then  ${\rm type}(G) \lneq {\rm type}(A)$ and so ${\rm type}(G) =[(\alpha_{p})]$  has at most one  $\infty$  and $\alpha_{p}=0$ for all but a finite  number $p$. Hence $G\cong \mathbb{Z}$ or $G\cong G_{p} $ and so by part (a), $\rm{W.dim}(G)=1$. Therefore, by definition $\rm{W.dim}(A)=2$.
Assume that the assertion is true for $n-1$ and let ${\rm type}(A)=[(\alpha_{p})]$, where $\alpha_{p}=0$ for all but a finite  number $p$ and for $n$ distinct prime number $p$,  $\alpha_{p}=\infty$. For any proper subgroupe $G$ of $A$ such that $A \not\rightarrowtail G$  we have  ${\rm type}(G) \lneq {\rm type}(A)$ and  so by induction hypothesis,  $\rm{W.dim}(G)=n-1$. Now by \Cref{rem1} we conclude that $\rm{W.dim}(A)=n$.

($\Rightarrow$). Assume that $\rm{W.dim}(A)=n \geq 2$ and consider the following cases:\\
\textbf{Case 1}: ${\rm type}(A)=[(\underbrace{0, 0,0, \ldots, 0}_{k-times}, m_{1}, m_{2}, m_{3}, \ldots )]$, where $k \in \mathbb{N} $ and $m_{i} \neq 0$ for any $i \geq 1$. Then
$ \langle  \frac{1}{ {p_{i_{1}}}^{m_{1}}}, \frac{1}{{p_{i_{2}}}^{m_{2}} }, \frac{1}{ {p_{i_{3}}}^{m_{3}} }, \cdots \rangle \supseteq \langle  \frac{1}{ {p_{i_{2}}}^{m_{2}} }, \frac{1}{ {p_{i_{3}}}^{m_{3}} }, 
\frac{1}{ {p_{i_{4}}}^{m_{4}} }, \cdots \rangle \supseteq \cdots $
 is a descending chain  of subgroups of $A$  that is not stopping up to monomorphism. Therefore by \Cref{1}, $A$ has not weakly uniserial dimension, a contradiction.\\
\textbf{Case 2}: If ${\rm type}(A)=[(\alpha_{p})]$ such that for $n+1$ distinct prime number $p$,  $\alpha_{p}=\infty$, then by ($\Leftarrow$) we have $\rm{W.dim}(A)=n +1$, a contradiction.\\
Therefore,  if $\rm{W.dim}(A)=n \geq 2$,  then ${\rm type}(A)=[(\alpha_{p})]$, where $\alpha_{p}=0$ for all but a finite  number $p$ and for $n$ distinct prime number $p$,  $\alpha_{p}=\infty$. 
\end{proof}

\begin{Cor}
Let $A$ be a torsion-free abelian group of ${\rm rank}$ $1$. Then $A$ has a finite weakly uniserial dimension.
\end{Cor}

\section{\hspace{-6mm}. Monoartinian (Mononoetherian) rings and modules}
In this section, we study monoartinian (mononoetherian)  modules  and rings and investigate some of their fundamental properties.

\begin{Def}
{\rm
Let $R$ be a ring and $M$ be a right $R$-module. We say that  $M$ is \textit{mononoetherian} if, for every ascending chain $M_{1} \subseteq M_{2} \subseteq  \cdots $ of submodules
of $M$, there exists an index $n \geq 1$ such that $M_{i}$ is embedded in $M_{n}$ for every $ i \geq n$.
}
\end{Def}

\begin{Lem}\label{mono-comp}
A module $M$ is monoartinian {\rm(}mononoetherian{\rm)} if and only if for every non-empty set $\mathcal{F}$ of submodules of $M$, there exists $N \in \mathcal{F}$  such that  for every  $N' \subseteq N$ {\rm(}$N \subseteq N'${\rm)}, if $N' \in \mathcal{F}$, then    $N \rightarrowtail N'$ {\rm(}$N' \rightarrowtail N${\rm)}. Equivalently, $M$ is monoartinian {\rm(}mononoetherian{\rm)}  if and only if  for every  non-empty chain $\mathcal{C}$ of submodules of $M$, there exists $N \in \mathcal{C}$  such that  for every  $N' \subseteq N$ {\rm(}$N \subseteq N'${\rm)},  if   $N' \in \mathcal{C}$, then   $N \rightarrowtail N'$ {\rm(}$N' \rightarrowtail N${\rm)}.
\end{Lem}

\begin{proof}
$(\Leftarrow)$. Let $M_{1} \supseteq M_{2} \supseteq \cdots$,
  be a descending chain of submodules of $M$ and $\mathcal{F}=\lbrace M_{1}, M_{2}, \ldots  \rbrace$
   is a non-empty  set of  submodules of  $M$. By hypothesis,  there exists $M_{i} \in \mathcal{F}$  which  for every $M_{j} \subseteq M_{i}$ such that $M_{j} \in \mathcal{F}$  we have    $M_{i} \rightarrowtail M_{j}$. This shows that  $M$ is monoartinian.
   
 $(\Rightarrow)$. Suppose that $M$  is monoartinian and $\mathcal{F}$ is a non-empty set of submodules  of  $M$. Let $M_{0}$  be an arbitrary element of $\mathcal{F}$. If for every $N \subseteq M_{0}$ which $N \in \mathcal{F}$ we have $M_{0} \rightarrowtail N$, then we are done. Assume that there exists  $M_{1} \subseteq M_{0}$ such that  
 $M_{1} \in \mathcal{F}$ and $M_{0} \not\rightarrowtail M_{1}$. If for every $N \subseteq M_{1}$ which $N \in \mathcal{F}$ we have $M_{1} \rightarrowtail N$, then we are done. Otherwise, there exists  $M_{2} \subseteq M_{1}$ such that  
 $M_{2} \in \mathcal{F}$ and $M_{1} \not\rightarrowtail M_{2}$. By continuing this process, either there exists $M_{i} \in \mathcal{F}$  such that  for every  $N \subseteq M_{i}$ which $N \in \mathcal{F}$, we have    $M_{i} \rightarrowtail N$ or there exists a descending chain $M_{0} \supsetneq M_{1} \supsetneq \cdots,$ of submodules of $M$ such that  $M_{i} \not\rightarrowtail M_{i+1}$, for every $i \geq 0$. The second case is a contradiction and this completes the proof.
\end{proof}

\begin{Rem}
{\rm
In  above lemma,  consider  $\mathcal{F}$ as the set of all non-zero submodules of $M$.  If  $M$  is a non-zero monoartinian module, then $M$ contains a non-zero submodule $L$ embedded in all its non-zero submodules. Thus $L$ is compressible.  Clearly every submodule of a monoartinian (mononoetherian) module  is monoartinian (mononoetherian).
}
\end{Rem}

To give some further examples, we need the following propositions.

\begin{Pro}
Let $M$ be a projective isosimple $R$-module and $N$ be a monoartinian
{\rm(}mononoetherian{\rm)} module. Then $N\oplus M$ is a monoartinian {\rm(}mononoetherian{\rm)}
$R$-module.
\end{Pro}

\begin{proof}
Let $M$ be a projective isosimple $R$-module and $N$ be a monoartinian module. Assume that 
$T_{1} \supseteq T_{2} \supseteq  \cdots$ is  a descending chain of submodules of  $M \oplus N$ and $p$ is the projection of $M \oplus N$ onto $M $ with kernel $ N$. denote by $p_{i}$ the restriction of $p$ to $T_{i}$ for every $i$.
Then ${\rm im}(p_{i}) =0$  or ${\rm im}(p_{i}) \cong M$, and so ${\rm im}(p_{i})$ is projective. Thus  $p_{i}$ is a splitting epimorphism and so there exists $\varepsilon_{i }: {\rm im}(p_{i}) \rightarrow T_{i}$  such that 
$T_{i} ={\rm ker}(p_{i}) \oplus \varepsilon_{i }({\rm im}(p_{i}))$. Now, considering that   $N \supseteq {\rm ker}(p_{1}) \supseteq {\rm ker}(p_{2}) \supseteq \cdots$ is stabilizes up to monomorphism  and   $\varepsilon_{i}$, in fact,  is a monomorphism,  we conclude that $T_{1} \supseteq T_{2} \supseteq  \cdots$ is stopping up to monomorphism, as desired.
\end{proof}

\begin{Examp}
\rm{(1)}
Every finitely generated $\mathbb{Z}$-module is monoartinian (mononoetherian).

\rm{(2)} The $\mathbb{Z}$-module $\mathbb{Z}_{p^{\infty}} \oplus \mathbb{Z}$ is monoartinian.
\end{Examp}

\begin{Pro}
Let $M$ be a finitely generated semisimple $R$-module and $N$ be a mononoetherian  
{\rm(}monoartinian{\rm)} $R$-module. Then $M\oplus N$ is a  mononoetherian   {\rm(}monoartinian{\rm)} $R$-module.
\end{Pro}

\begin{proof}
Assume that  $N$ is a mononoetherian  $R$-module and $M=\oplus_{i=1}^{n} S_{i}$,  where $S_{i}$'s are simple $R$-module. By induction, it is sufficient to prove the assertion for $S \oplus N$,  where $S$  is a simple $R$-module. Let $K_{1} \subseteq K_{2} \subseteq \cdots$ be an ascending chain of submodules of  $S \oplus N$. If  $K_{i} \cap S =0$ for all $i$, then $K_{i}$'s  are submodules of $N$ and so we are done. If $K_{i} \cap S \neq 0$ for some $i$, then $ S \subseteq K_{i}$ and so $ S \subseteq K_{j}$ for any $j \geq i$. 
Now, consider the ascending chain
$K_{j} \cap N \subseteq K_{j+1} \cap N \subseteq \cdots $ of submodules of $N$. Since  $N$ is mononoetherian, there exists $l  \geq j$  such that   $K_{t} \cap N  \rightarrowtail K_{l} \cap N $ for any $t \geq l$.
Therefore
$$K_{t}=K_{t} \cap (S \oplus N)=S \oplus (K_{t} \cap N) \rightarrowtail S \oplus (K_{l} \cap N) =K_{l} \cap (S \oplus N) =K_{l},$$ as desired.
\end{proof}

\begin{Rem}
\rm {
 The finitely generated condition  in above proposition is necessary. For example the $\mathbb{Z}$-module $(\oplus_{p \in \mathbb{P}} \mathbb{Z}_{p}) \oplus \mathbb{Z}$ where $P$ is the set of prime numbers is not mononoetherian (monoartinian).}
\end{Rem}

\begin{Def}
\rm{ We say that a ring $R$ is right monoartinian {\rm(}mononoetherian{\rm)} if the right
$R$-module $R_{R}$ is monoartinian {\rm(}mononoetherian{\rm)}.}
\end{Def}

\begin{Pro}\label{upper}
Let $R$ and $S$ be two rings  and $_{R}M_{S}$ be a bimodule. The ring
 $T= \left[
         \begin{array}{rr}
              R  & M \\
              0 &  S 
          \end{array} \right]$
is right monoartinian {\rm(}right mononoetherian{\rm)} if and only if the right modules
$(M \oplus S)_{S}$
 and $R_{R}$ are monoartinian {\rm(}mononoetherian{\rm)}.
\end{Pro}

\begin{proof}
 The proof follows  from the fact that every right ideal of $T$ is of the
form $J_{1} \oplus J_{2} $, where $J_{1} $ is right ideal of $R$, $J_{2}$ is a submodule of $(M \oplus S)_{S}$ and
$J_{1}M \subseteq J_{2}$, by \cite[Proposition 1.17]{Lam}.
\end{proof}

The following example shows that finitely generated modules over a right monoartinian ring  are not necessarily monoartinian. 

\begin{Examp}
{\rm
In the notation of \Cref{upper} and its proof, set
$M =\mathbb{Z}^{(\mathbb{P})}$ (a direct sum of copies of the additive group $\mathbb{Z}$ of integers indexed in the set of all prime numbers $P$) and $R = S = \mathbb{Z}$. By \cite[Proposition 2.12]{Shi} and \Cref{1}, $(M \oplus S)_{S}$ is monoartinian.  Also, clearly 
$R_{R}$ is monoartinian.  Then $T$ is a right monoartinian ring. Now, set $J_{1} = 0$ and
$J_{2}=\oplus_{p \in \mathbb{P}} p\mathbb{Z}$. Then $J_{1} \oplus J_{2}$ is a right ideal of $T$ and the factor module $T/J_{1} \oplus J_{2}$ is not monoartinian, because $\oplus_{p \in \mathbb{P}} \mathbb{Z}/p\mathbb{Z} $ is not monoartinian. This ring $T$ also shows that there
are right monoartinian rings that are not right mononoetherian.
}
\end{Examp}

\begin{Lem}\label{morita}
Being monoartinian {\rm(}mononoetherian{\rm)} is a Morita invariant  property of modules.
\end{Lem}

\begin{proof}
The proof follows from \cite[Proposition 21.2]{AF}.
\end{proof}

Similar to the proof of \cite[Theorem 2.1]{Prak},  we have the following proposition.
\begin{Pro}\label{poly}
Let $R$ be a right mononoetherian {\rm(}monoartinian{\rm)} ring. Then $R\left[ x \right]  $  is right mononoetherian {\rm(}monoartinian{\rm)}.
\end{Pro}

\begin{proof}
Assume that $I$ be a right ideal of  $R\left[ x \right]  $. Then define $\phi_{i}(I)$ to be the set of all coefficients of thoes elements of $I$ whoes degree is less than or equal to $i$, $i \geq 0$. For any $i \geq 0$, it is easy to see that $\phi_{i}(I)$ is a right ideal of $R$. By definition, it is clear that if $I_{n} \subseteq I_{m}$, then $\phi_{i}(I_{n}) \subseteq \phi_{i}(I_{m}) $, for any $i \geq 0$. Now, let $I_{0} \subseteq I_{1} \subseteq I_{2} \subseteq \cdots$
 be an ascending chain of right ideals of  $R\left[x\right] $. Then for any $i \geq 0$, $\phi_{i}(I_{0}) \subseteq \phi_{i}(I_{1}) \subseteq \phi_{i}(I_{2}) \subseteq \cdots$
and   $\phi_{0}(I_{0}) \subseteq \phi_{1}(I_{1}) \subseteq \phi_{2}(I_{2}) \subseteq \cdots$
are ascending chains of right ideals of $R$. Since $R$ is right mononoetherian, there exists $p \geq 0$ such that  for any  $t \geq p$, $\phi_{t}(I_{t}) \rightarrowtail \phi_{p}(I_{p})$. Thus we can say that  for any  $i \geq p$ and $j\geq 0$,  $\phi_{i}(I_{j}) \rightarrowtail \phi_{p}(I_{p})$. Also since 
 $\phi_{i}(I_{0}) \subseteq \phi_{i}(I_{1}) \subseteq \phi_{i}(I_{2}) \subseteq \cdots$, is an ascending chain in $R$,  for any $i \geq 0$, there exists  $k_{i}\geq 0$ such that  $\phi_{i}(I_{l}) \rightarrowtail \phi_{i}(I_{k_{i}})$,  for every $l \geq k_{i}$. If  $i \geq p$, then $k_{i}=p$ will work. Thus the set $\lbrace k_{i}\rbrace$ is bounded. Let $k$ be an upper bound of $\lbrace k_{i}\rbrace$. Then   $\phi_{i}(I_{j}) \rightarrowtail \phi_{i}(I_{k})$, for any  $i \geq 0$ and $j \geq k$. Suppose that $\psi_{nk}:\phi_{n}(I_{k+1}) \rightarrow \phi_{n}(I_{k})$ is a monomorphism. For $f(x)=a_{0}+a_{1}x+\cdots+a_{n}x^{n}$, define $\phi: I_{k+1} \rightarrow I_{k}$ by 
 $\phi(f(x))=\psi_{nk}(a_{0})+\psi_{nk}(a_{1})x+ \ldots+\psi_{nk}(a_{n})x^{n}$. It is easy to see that  $\phi$ is a right  $R\left[ x \right] $-monomorphism, as desired.
\end{proof}

\begin{The}
Any mono-artinian module $M$ contains an essential submodule that is a direct sum of compressible modules.
\end{The}
\begin{proof}
Let  $\mathcal{S}$  denote the set of all families of independent compressible  submodules of  $M$. By Zorn's Lemma, $\mathcal{S}$ has a maximal member  $W= \lbrace  V_{\lambda}~ | ~\lambda \in \Lambda \rbrace$. Since every non-zero submodule of $M$ contains a compressible  module, $V=\oplus_{\lambda \in \Lambda}V_{\lambda}$ is essential  in  $M$.  Otherwise  there exists $N \subseteq M$  such that  $V \cap N=0$ and so   $V \oplus N \subseteq M$, which contradicts the maximality  of   $W$.
\end{proof}

\begin{Lem}
Let $R$  be a commutative ring and $M$ be a  finitely generated mono-artinian module. If  $M\cong M\oplus N$ for some  module  $N$, then $N$ is noetherian.
\end{Lem}
\begin{proof}
If $N_{1}$ is a submodule of $N$ that is not finitely generated, then $M$ contains a submodule isomorphic to $M \oplus N_{1}$, say  $K_{1}$, which is not  finitely generated. Since  $K_{1} \cong M \oplus N_{1}$, there exists $K_{2} \subseteq K_{1}$ such that    $K_{2}\cong M$. Containing in this manner, we have    
$M \supseteq K_{1} \supseteq K_{2} \supseteq \cdots$,
  such that  for odd indices, $K_{i}\cong  M \oplus N_{1}$  is not finitely generated and for  even indices, $K_{i} \cong M$  is finitely generated. Since $M$ is mono-artinian, there exists an odd  index $n$ such that   $K_{n} \rightarrowtail K_{n+1}$ and so  $M \oplus N_{1} \overset{f}\rightarrowtail M$. Now,  \cite[Exercise 1.10]{Lam} implies that $M \oplus N_{1} \cong M$,  a contradiction.
\end{proof}

\begin{Pro}
Let $R$ be a right non-singular ring with maximal right quotient ring $Q$. Let
$M$ be a right $Q$-module that is a non-singular $R$-module and an monoartinian (mononoetherian)
$R$-module. Then $M$ is monoartinian (mononoetherian) as a $Q $-module.
\end{Pro}

\begin{proof}
Let $M \supseteq M_{1} \supseteq M_{2} \supseteq \cdots $ be a descending chain of $Q$-submodules of $M$. Then this
chain is also a descending chain of $R$-submodules of $M$ and thus there exists $n \geq 1$ such that $M_{n} \overset{f_{k}}\rightarrowtail M_{k}$ as $R$-modules for any $k \geq n$.
 If $q \in Q$, then there exists an essential
right ideal $I$ of $R$ such that $qI \subseteq R$. Then, for every $t \in M_{n}$ and $e \in I$, we have 
$f_{k}(tqe) =f_{k}(tq)e$
and $f_{k}(tqe) = f_{k}(t)qe$, so $(f_{k}(tq) - f_{k}(t)q)e = 0$, i.e., $(f_{k}(tq) - f_{k}(t)q)I= 0$. Since $M$ is
non-singular,   $f_{k}(tq) = f_{k}(t)q$. Therefore $f_{k}$ is a $Q$-monomorphism. 
\end{proof}

\begin{Cor}
If $R$ is a commutative domain  with the field of fractions $Q$ and $M$ is a  monoartinian  $R $-module, then  $M$ is  isoartinian as a $Q $-module.
\end{Cor}

\begin{Pro}\label{direct product}
Let $R$ be a ring and $M=\oplus_{1}^{n} M_{i}$, where $M_{i}$'s are right  $R$-modules and $M$ is a duo module. Then $M$ is mononoetherian (monoartinian) if and only if  each $M_{i}$ is mononoetherian (monoartinian). 
\end{Pro}
\begin{proof}
Assume that  every $M_{i}$ is mononoetherian and $N_{1} \subseteq N_{2} \subseteq \cdots$ is an ascending chain of submodules  of $M$. Since $M$ is duo, $N_{i} = \oplus_{j=1}^{n} (N_{i} \cap M_{j}) $ for each $i$ and so for any $1 \leq j \leq n$ there exists $k \geq 1$ such that $N_{l} \cap M_{j} \rightarrowtail N_{k} \cap M_{j}$, for every $l \geq k$. Therefore, $N_{l} = \oplus_{j=1}^{n} (N_{l} \cap M_{j}) \rightarrowtail  \oplus_{j=1}^{n} (N_{k} \cap M_{j}) = N_{k} $, i.e, $M$ is mononoetherian. Converse is clear.
\end{proof}

\begin{Rem}
\rm{
 The finite condition  in above proposition is necessary. For example the $\mathbb{Z}$-module $\oplus_{p \in P} \mathbb{Z}_{p^{\infty}}$ is not mononoetherian (monoartinian).}
\end{Rem}

\begin{Pro}
A right monoartinian ring $R$ is semiprime  if and only if  the  intersection of the annihilators  of the compressible  right  $R$-modules is zero. 
\end{Pro}
\begin{proof}
Let $I$ be the intersection of the  annihilators of all compressible  right  $R$-modules. Assume that $R$ is semiprime  and  $I \neq 0$.  Since $I_{R}$  is monoartinian, it contains a compressible  right ideal  $C$  of $R$. Then $C \subseteq {\rm ann}(C)$ and so   $C=0$, a contradiction.  For the converse, suppose that   $I=0$ and $K$ is  an ideal of $R$ such that $K^{2}=0$. If $M$  is a compressible right module such that  $MK \neq 0$, then  $M \overset{f}\rightarrowtail MK$ and so  $M \cong f(M) \subseteq MK$. Thus  $f(M) K \subseteq MK^{2}=0$ and so  $f(MK)=0$, a contradiction. Therefore, $MK=0$  and so  $K \subseteq I=0$.
\end{proof}

\begin{Pro}\label{nonsingular}
A semiprime right  monoartinian ring is right non-singular.
\end{Pro}
\begin{proof}
Suppose that   $\mathcal{Z}(R_{R}) \neq 0$. Then there exists  a compressible right  ideal    $C \subseteq \mathcal{Z}(R_{R})$. If   $C^{2} \neq 0$, then  there exists    $x \in C$ such that  $xC \neq 0$  and so    $xC \cong C \slash {\rm r.ann}(x) \cap C$.  If   ${\rm r.ann}(x) \cap C \neq 0$, then $C \overset{f}\rightarrowtail {\rm r.ann}(x) \cap C$  and  so    $C \cong f(C) \subseteq {\rm r.ann}(x) \cap C \subseteq {\rm r.ann}(x)$. Then   $xf(C)=0$ and so $(f(C) x)^{2}=0$.  This implies that $f(C)x=0$, since $R$ is semiprime. Hence  $Cx=0$ and so $(xC)^{2}=0$. This implies that  $xC=0$, a contradiction. Hence $ {\rm r.ann}(x) \cap C =0$ which is a contradiction,
 since   ${\rm r.ann}(x) \subseteq_{e} R_{R}$. Therefore,   $\mathcal{Z}(R_{R}) = 0$.
\end{proof}

The following corollary that state in  \cite[Lemma 4.3]{Fa}, immediately follows from above proposition.

\begin{Cor}
A semiprime right isoartinian ring is right non-singular.
\end{Cor}

\begin{Lem}
Let $R$ be a semiprime ring and $C$ be a  compressible  right ideal of  $R$.  If $I_{C}$ is the sum of all right ideals  of $R$ isomorphic to  $C$, then $I_{C}$ is a two-sided  ideal  of $R$.
\end{Lem}
\begin{proof}
Let $x \in R$ and  $J \subseteq R_{R}$ such that   $J \cong C$. To show that  $I_{C}$  is a two-sided  ideal of $R$, it is suffices to prove that   $xJ \subseteq I_{C}$. If   $xJ \neq 0$, then  clearly   $xJ \cong J \slash {\rm r.ann}(x)  \cap J $.  we claim that   $ {\rm r.ann}(x)  \cap J=0$. for if not, $J \overset{f}\rightarrowtail  {\rm r.ann}(x)  \cap J $ and so $  J \cong f(J) \subseteq {\rm r.ann}(x)$. Thus $xf(J)=0$ and so $(f(J) x)^{2}=0$. Hence $f(J) x=0$, since $R$ is semiprime. This implies that $Jx=0$ and so $(xJ)^{2}=0$. This implies that $xJ=0$,  contradiction. Therefore, $ {\rm r.ann}(x)  \cap J=0$  and so $xJ \cong J$, i.e, $xJ \subseteq I_{C}$.
\end{proof}

If $A$ is an ideal in a semiprime ring $R$, then $ {\rm r.ann}(A)= {\rm l.ann}(A)$ ($=  {\rm ann}(A)$, say).
By annihilator ideal, we mean an ideal of form $ {\rm ann}(A)$ for some ideal $A$ of $R$. See \cite[Proposition 2.2.14, p. 55]{Mc}.

\begin{Lem}
Let $R$  be a semiprime right monoartinian ring. Then $R$ satisfies the Acc on annihilator ideals.
\end{Lem}

\begin{proof}
Let $A_{1} \subseteq A_{2} \subseteq \cdots$ be an ascending chain of annihilator ideals. Then we have a descending chain ${\rm ann}(A_{1}) \supseteq {\rm ann}(A_{2}) \supseteq \cdots$.  Since $R$ is  right monoartinian, there exists $n \geq 1$ such that ${\rm ann}(A_{n}) \rightarrowtail {\rm ann}(A_{k})$, for all $k \geq n$. Thus ${\rm ann}({\rm ann}(A_{k})) \subseteq {\rm ann}({\rm ann}(A_{n}))$ and so $A_{k} \subseteq  A_{n}$. On the other hands, $A_{n} \subseteq  A_{k}$. Therefore $A_{n} =  A_{k}$, for all $k \geq n$. 
\end{proof}

\begin{Lem}\label{Acc on anihilators}
Let $R$  be a commutative ring. Then the following statements hold:\\
\indent{\rm{(a)}} If $ R$ is monoartinian {\rm(}mononoetherian{\rm)}, then it satisfies the Acc {\rm(}Dcc{\rm)} on annihilators.

\indent{\rm{(b)}} If $R$ is self-injective, then $R$ is monoartinian {\rm(}mononoetherian{\rm)} if and only if $R$ is artinian  \indent\indent{\rm(}noetherian{\rm)}.
\end{Lem}

\begin{proof}
(a) Let $I_{1} \subseteq I_{2} \subseteq \cdots$ be an ascending chain of ideals of $R$ that are annihilators of
subsets of $R$. Then $ {\rm ann}(I_{1}) \supseteq {\rm ann}(I_{2}) \supseteq \cdots$ .
Since $R$ is monoartinian, there exists $n \geq 1$ such that 
${\rm ann}(I_{n}) \rightarrowtail {\rm ann}(I_{i})$
for any $ i \geq n$. Taking
the annihilators of these annihilators ideals, we get that $I_{i} \subseteq I_{n}$. Therefore $I_{n} = I_{i}$, for any $i \geq n$.

(b) follows from (a) and the fact that a right self-injective ring with the ascending
chain condition on right annihilators is QF.
\end{proof}

\begin{Pro}\label{Goldie}
Every commutative  semiprime monoartinian ring is a Goldie  ring.
\end{Pro}
\begin{proof}
By \Cref{nonsingular}, $R$ is a non-singular ring and by \Cref{Acc on anihilators}, $R$ has Acc annihilator  ideals. Then \cite[Theorem 11.43]{Lam}, implies that $R$ has Acc on complements and so \cite[Theorem (6.30)']{Lam}, implies that 
 ${\rm u.dim}(R_{R})<\infty$. Therefore, by  \cite[Theorem 11.13]{Lam}, $R$ is a Goldie ring.
\end{proof}

\begin{Rem}
{\rm The commutative condition is necessary in above proposition. For instance consider $R=K \langle x_{1}, x_{2}\rangle$, where $K$ is a domain. Then $R$  is a non-commutative domain that is not right Goldie.}
\end{Rem}

\begin{Pro}\label{u.dim}
Let  $R$  be a commutative ring. Then every mononoetherian  $R$-module has finite uniform dimension.
\end{Pro}
\begin{proof}
Let $M$ be a mononoetherian $R$-module. If $\oplus_{i=1}^{\infty}M_{i}$ is a submodule of $M$, where each
$M_{i}$ is non-zero, we can choose a non-zero element $a_{i} \in M_{i}$ for each $i$. Then $\oplus_{i=1}^{\infty}a_{i} R$ is a mononoetherian module. Consider the ascending chain $a_{1}R \subseteq a_{1}R \oplus a_{2}R \subseteq \cdots $. Then there exists $n \geq 1$ such that  $ a_{1}R \oplus \cdots \oplus a_{n+1}R \overset{f}\rightarrowtail  a_{1}R \oplus \cdots \oplus a_{n}R $ and so $ a_{1}R \oplus \cdots \oplus a_{n+1}R \cong f( a_{1}R \oplus \cdots \oplus a_{n+1}R ) \subseteq  a_{1}R \oplus \cdots \oplus a_{n}R $. 
 This implies that, for some $1 \leq j \leq n$, $a_{j}R$ is not noetherian. For if not, $\sum_{i=1}^{n+1} {\rm u.dim}(a_{i}R) \leq \sum_{i=1}^{n} {\rm u.dim}(a_{i}R)$, a contradiction.
 Set $b_{1}=a_{j}$. Similarly, $\oplus_{i=n+1}^{\infty}a_{i} R$ is a mononoetherian and so there exists  $k \geq n+1$  such that $a_{k}R$ is nonnoetherian. Set $b_{2}=a_{k} $. Continuing in this manner, we find a sequence $b_{1}, b_{2}, \ldots$ such that $\oplus_{i=1}^{\infty}b_{i} R$ is  mononoetherian  and each $b_{i} R$ is not noetherian. For each
 $i \geq 1$,
let $K_{i}$ be a nonfinitely generated submodule of $b_{i} R$ and consider the ascending chain
$K_{1} \subseteq b_{1} R \subseteq b_{1} R \oplus K_{2} \subseteq b_{1} R \oplus b_{2} R \subseteq  b_{1} R \oplus b_{2} R \oplus K_{3} \subseteq \cdots$
 of submodules of $\oplus_{i=1}^{\infty}b_{i} R$. Then there exists  $t \geq 1$ such that
 $ b_{1} R \oplus b_{2} R \oplus  \cdots b_{t}R\oplus K_{t+1} \rightarrowtail b_{1} R \oplus b_{2} R \oplus  \cdots b_{t}R$. Also $\pi_{t}: b_{1} R \oplus b_{2} R \oplus  \cdots b_{t}R\oplus K_{t+1}  \rightarrow b_{1} R \oplus b_{2} R \oplus  \cdots b_{t}R$ is an epimorphism. Therefore, $b_{1} R \oplus b_{2} R \oplus  \cdots b_{t}R\oplus K_{t+1} \cong b_{1} R \oplus b_{2} R \oplus  \cdots b_{t} R$.
This implies that $K_{n+1}$ is finitely generated,  a contradiction. 
\end{proof}

\begin{Cor}
If $D$ is a commutative mononoetherian domain, then $D$ is an Ore domain.
\end{Cor}

\begin{Lem}{\rm (see \cite[Lemma 6.3]{Krull})}\label{krull}
Applied to a non-zero right ideal of a semiprime ring with {\rm Krull} dimension, the adjectives uniform, monoform, critical, and 
compressible are equivalent. 
\end{Lem}

\begin{Pro}
Let $R$ be a semiprime ring that has {\rm Krull} dimension. If every non-zero right ideal of $R$ contains a right regular element, then $R_{R}$ is monoartinian if and only if $R_{R}$  is uniform. 
\end{Pro}
\begin{proof}
$(\Rightarrow)$. Since $R_{R}$ is monoartinian, it has a compressible  right ideal, say $I$. By hypothesis, there exists a non-zero right regular element $x \in I$. Thus $R \cong xR$,  $I \cong xI$ and we have a descending chain
$R \supseteq I \supseteq xR \supseteq xI \supseteq x^{2}R \supseteq x^{2}I \supseteq \cdots$ 
of submodules of $R$.  Since $R_{R}$ is monoartinian, there exists $n \geq 1$ such that $x^{n}R \rightarrowtail x^{n}I$. Then  $R \overset{f}\rightarrowtail I$ and so $R \cong f(R) \subseteq I$. Now by \Cref{krull}, $R_{R}$ is uniform.

$(\Leftarrow)$. Let $R_{R}$  be uniform and $I_{1} \supseteq I_{2} \supseteq \cdots$ be a descending chain of its right ideals.  Then every $I_{i}$ is uniform and so by \Cref{krull}, is compressible. Therefore, $R_{R}$ is monoartinian.
\end{proof}

\begin{Pro}
Let $R$ be a semiprime local ring with ${\rm Soc}(R_{R})\neq 0$. Then every right $R$-module is monoartinian  if and only if $R$ is a homogeneous semisimple ring.
\end{Pro}

 \begin{proof}
 One side is clear. For the other side, assume that every right $R$-module is monoartinian. Then, by \Cref{nonsingular}, $R$ is  a right non-singular ring  and so $\mathcal{Z}(I) = 0$, for every minimal right ideal $I$ of $R$. On the other hands we know that every simple module is either singular or projective, but not both. Thus $I$ and so every simple right $R$-module is projective. Therefore, $R$ is a homogeneous semisimple ring.
 \end{proof}

Recall that a module $M$ is saide to be essentially retractable if ${\rm Hom}_{R}(M, N) \neq 0$ for all
essential submodules $N$ of $M$. 

\begin{Pro}\label{essen ret}
Let  $M$ be a right $R$-module. If $\oplus_{1}^{\infty} M$  is  monoartinian, then $M$ is an essentially retractable module.
\end{Pro}

\begin{proof}
Assume that $N$ is an essential submodule of  $M$. Consider $N$ as a submodule of   $\oplus_{1}^{\infty} M$. Then  
$\oplus_{1}^{\infty} M \supseteq N\oplus (\oplus_{2}^{\infty} M) \supseteq \oplus_{2}^{\infty} M \supseteq N\oplus (\oplus_{3}^{\infty} M) \supseteq  \cdots $,
is a descending chain of submodules of  $\oplus_{1}^{\infty} M$ and so by hypothesis there exists $n \geq 1$  such that for any $k \geq n$ we have a non-zero $R$-monomorphism  $f: \oplus_{k}^{\infty} M \rightarrow N \oplus (\oplus_{k+1}^{\infty} M)$. Hence 
$M \overset{\iota}\hookrightarrow \oplus_{k}^{\infty} M \overset{f}\rightarrowtail  N \oplus (\oplus_{k+1}^{\infty} M) \overset{\pi}\rightarrow N $, where $\pi$  is a  suitable canonical projection,
  is a non-zero $R$-homomorphism  from $M$ to $N$, as desired.
\end{proof}

The following corollary shows that if  $M$ is a uniform injective $\mathbb{Z}$-module, then $\oplus_{1}^{\infty}  M$ is not monoartinian.

\begin{Cor}
Let $R$ be a right hereditary ring and $M$ be a uniform right $R$-module. Then\\
\indent{{\rm(a)}}. Assume that $M$ is an injective module. $\oplus_{1}^{\infty}  M$  is monoartinian if and only if $M$  is a \indent\indent simple module.\\
\indent{{\rm(b)}}. Suppose that $M$ is a projective module. If  $\oplus_{1}^{\infty} M$  is monoartinian, then $M$ is a  \indent\indent compressible  module.
\end{Cor}

\begin{proof}
(a).$(\Leftarrow)$ is clear.\\
 $(\Rightarrow)$. Let $\oplus_{1}^{\infty} M$  be monoartinian. Then, by \Cref{essen ret}, there exists a non-zero $R$-homomorphism $f: M \rightarrow N$, for any non-zero submodule $N$ of  $M$. Now, since $R$ is  right hereditary and $M$ is  injective, we conclude that $N=M$.
 
 (b). Assume that $\oplus_{1}^{\infty} M$  is monoartinian. Then there exists a non-zero $R$-homomorphism $f: M \rightarrow N$, for any non-zero submodule $N$ of  $M$. By hypothesis, $M/ {\rm ker}(f)$  is projective and so the short exact sequence  $0 \rightarrow   {\rm ker}(f) \rightarrow M  \rightarrow M/ {\rm ker}(f) \rightarrow 0$, splits.  Hence
 $M \cong {\rm ker}(f) \oplus M/ {\rm ker}(f)$. But $f(M) \neq 0$ and  $M$ is uniform, therefore  ${\rm ker}(f) =0$, as desired.
\end{proof}

Recall that a ring is said to be \textit{right  mod-retractable}  if  over which every  right module is retractable. 
Also, a ring is called \textit{right max} provided every non-zero right module contains a maximal submodule, (for more details see \cite{Tam}).

\begin{Pro}\label{a}
{\rm (\cite[Proposition 3.1]{Tam})} $R$ is a right mod-retractable ring if and only if, for every non-zero
module $M$ and every $m \in M$ such that $mR \subseteq_{e} M$, there exists a non-zero homomorphism
$M \rightarrow mR$.
\end{Pro}
 
\begin{The}\label{b}
{\rm (\cite[Theorem 3.3]{Tam})} If $R$ is a right mod-retractable ring, then R is right max.
\end{The}

\begin{Lem}\label{c}
{\rm (\cite[Lemma 3.2]{Tam})} If $R$ is a right max ring, then ${\rm J}(R)$ is right T-nilpotent.
\end{Lem}

\begin{Pro}\label{perfect}
Let $R$ be a ring such that every right $R$-module is monoartinian. Then $R$ is a right perfect ring.
\end{Pro}

\begin{proof}
By \Cref{first every}(b), $R$ is a semilocal ring, i.e, $R/J(R)$  is  semisimple. Since $\oplus_{1}^{\infty} M$ is monoartinian, for every right $R$-module   $M$, $M$ is an essentially  retractable module. Hence  \Cref{a}, implies that   $R$ is a right mod-retractable ring. Thus  by \Cref{b}, R is right max and so by \Cref{c},  ${\rm J}(R)$ is right T-nilpotent. Therefore, $R$ is a right perfect ring.
\end{proof}

\begin{Cor}
 Let $R$ be a right noetherian ring such that every right $R$-module is monoartinian.  Then $R$ is a right artinian ring.
\end{Cor}

\begin{proof}
The proof is clear, by \Cref{perfect} and \ref{first every}(c).
\end{proof}

\end{document}